\documentclass[12pt]{amsart}
\usepackage{ fullpage, amsmath, amsthm, amssymb, enumerate, amsfonts, hyperref, xypic, hyperref}

\def\ZZ{{\mathbb Z}}
\def\NN{{\mathbb N}}

\def\CC{{\mathbb C}}

\def\curlyF{{\mathcal F}}

\def\RR{{\mathbb R}}
\def\QQ{{\mathbb Q}}
\def\Qbar{{\overline\mathbb Q}}

\def\Q{{\mathbb Q}}
\def\Qbar{{\overline{\mathbb Q}}}
\def\Qtr{{{\mathbb Q^{\operatorname{tr}}}}}
\def\Ztr{{{\mathbb Z^{\operatorname{tr}}}}}
\def\Qab{{{\mathbb Q^{\operatorname{ab}}}}}

\def\OO{{\mathcal O}}

\DeclareMathOperator{\tr}{tr}

\DeclareMathOperator{\ab}{ab}

\DeclareMathOperator{\JR}{JR}

\newtheorem{theorem}{Theorem}[section]
\newtheorem*{theorem*}{Main Theorem}

\newtheorem{lemma}[theorem]{Lemma}
\newtheorem{prop}[theorem]{Proposition}
\newtheorem{corollary}[theorem]{Corollary}
\newtheorem{definition}[theorem]{Definition}

\newtheorem{remark}[theorem]{Remark}

\usepackage[dvipsnames]{xcolor}

\author{Caleb Springer}
\address{Department of Mathematics, University College London, Gower Street, London,
UK\\
and The Heilbronn Institute for Mathematical Research, Bristol, UK}
\email{c.springer@ucl.ac.uk}
\title{Definability and decidability for rings of integers in totally imaginary fields}
\date{\today}

\begin{document}

\maketitle
\begin{abstract}
We show that the ring of integers of $\Qtr$
 is existentially definable in the ring of integers of $\Qtr(i)$, where $\Qtr$ denotes the field of all totally real numbers. 
This implies that the ring of integers of $\Qtr(i)$ is undecidable and first-order non-definable in $\Qtr(i)$.
More generally, when $ L$ is a totally imaginary quadratic extension of a totally real field $K$,
we use the unit groups $R^\times$ of orders $R\subseteq \OO_L$ to produce existentially definable totally real subsets $X\subseteq \OO_L$. Under certain conditions on $K$, including the so-called $\JR$-number of $\OO_K$ being the minimal value $\JR(\OO_K) = 4$, we deduce the undecidability of $\OO_L$. 
  This extends previous work which proved an analogous result in the opposite case $\JR(\OO_K) = \infty$.
  In particular, unlike prior work, we do not require that $L$ contains only finitely many roots of unity.
\end{abstract}

\section{Introduction}

\subsection{A motivating example}
This paper is motivated by the desire to prove the following theorem concerning definability and decidability. For background and an overview of decidability and definability for infinite algebraic extensions of $\QQ$, we refer readers to the introduction of \cite{shlap}; see also the prequel to our work here \cite{Springer20}.
Recall that $\alpha\in \Qbar$ is \emph{totally real} if the  roots of its minimal polynomial are all real numbers. 
We write $\Qtr$ for the field of all totally real algebraic numbers and $\Ztr$ for its ring of integers, and we set $i = \sqrt{-1}$.
\begin{theorem} 
\label{thm:intro_motivation}
The ring of integers $\Ztr$ of $\Qtr$ is existentially definable in the ring of integers of $\Qtr(i)$. 
In particular, the first-order theory of the ring of integers of $\Qtr(i)$ is undecidable.
 \end{theorem}
 
 \begin{remark} 
 \label{rem:Razon}
\normalfont
 The paragraph below Corollary 3.19 in \cite{Koenigsmann14} mistakenly declared the ring of integers of $\QQ^{\tr}(i)$ to be decidable.\footnote{We thank Aharon Razon for pointing this out after reading an earlier draft of this paper.} 
 The reference given for this claim is \cite[Theorem~10.7]{Darniere00}, which actually proves the decidability of the overring $\OO_{\QQ^{tr}(i)}[\frac{1}{p}]$ where $p$ is any prime number. As seen above, the ring of integers $\OO_{\QQ^{tr}(i)}$ itself is undecidable.
  \end{remark}
  
Theorem \ref{thm:intro_motivation} is proven by extending and generalizing the methods which were first used to prove the {first-order} undecidability of rings of integers in fields such as $\QQ^{(2)}$  \cite{mruv} and $\QQ^{(d)}_{\text{ab}}$ for $d\geq 2$ \cite{Springer20}. Here, $F^{(d)}$ denotes the compositum of all extensions of degree at most $d$ of a field $F$, and $F^{(d)}_{\text{ab}}\subseteq F^{(d)}$ is the maximal abelian subfield.
   However, notice that the papers \cite{mruv, Springer20} additionally proved that the fields $\QQ^{(d)}_{\text{ab}}$ themselves  are undecidable by appealing to results of Videla \cite{vid} and Shlapentokh \cite{shlap} which imply that each $\QQ^{(d)}_{\text{ab}}$ has a first-order definable ring of integers.
   
The situation is entirely the opposite in the setting of Theorem \ref{thm:intro_motivation}. Indeed, Fried, Haran, and V\"{o}lklein proved that $\Qtr$ is decidable \cite{fhv}, which implies that $\Qtr(i)$ is also decidable. 
    It has been previously observed that, because J. Robinson proved that $\Ztr$ is undecidable \cite{jr}, the field $\Qtr$ is thus an example of a subfield of $\Qbar$ whose ring of integers $\Ztr$ does not admit any first-order definition in $\Qtr$. In the framework of Shlapentokh \cite[\S2.1]{shlap}, this is intuitively understood as a result of the fact that the field $\Qtr$ is too ``close'' to $\Qbar$, so it does not have enough ``expressive power'' to have a definable ring of integers. By Theorem~\ref{thm:intro_motivation}, we immediately see that $\Qtr(i)$ also has a first-order non-definable ring of integers, as expected.

 We pause to note that $\Qtr(i)$ is known to be an $\omega$-free PAC field; see \cite[Example 5.10.7]{Jarden11}. 
 Therefore, the non-definability of the ring of integers of $\Qtr(i)$ can also be deduced from the following result of Dittman and Fehm: If $R$ is a first-order definable subring of an $\omega$-free PAC field $L$, then $R$ is a field  \cite[Proposition 3]{DF21}.
 Their method, which builds upon work of Chatzidakis \cite{Chatzidakis19}, is entirely different from the techniques that we use below.

\subsection{Main results}
\label{sec:main_results}
In general, this paper considers totally imaginary quadratic extensions of totally real fields. We begin with a definability result.
Say that a totally real field $K\subseteq \Qbar$ is \emph{closed under square roots} if $\alpha\in K$ whenever $\alpha$ is a totally real number such that $\alpha^2\in K$.

   \begin{theorem}[Theorem \ref{thm:def_tr}]
    If $K$ is a totally real field which is closed under square roots and $L$ is any quadratic totally imaginary extension of $K$, then $\OO_K$ is existentially definable in $\OO_L$.
    \end{theorem}

This theorem implies Theorem \ref{thm:intro_motivation} immediately, given J. Robinson's proof that $\Ztr$ is undecidable \cite{jr}. However, we can actually prove first-order undecidability for rings of integers $\OO_L$ in a more general context where $\OO_K$ is not known to be existentially definable in $\OO_L$. For any $n\geq1$, let $\zeta_n$ denote a primitive $n$-th root of unity.

\begin{theorem}[Theorem \ref{thm:undec_JR4}]
\label{thm:intro_undec_JR4}
    Let $S\subseteq \NN$ be infinite. Define $K_0 = \QQ(\{\zeta_n + \overline\zeta_n : n\in S\})$ and let $K_1$ be the maximal totally real subfield of $K_0^{(2)}$. If $K\supseteq K_1$ is a totally real field and $L$ is any totally imaginary quadratic extension of $K$, then $\OO_L$ is undecidable.
  \end{theorem}

In this theorem, the  so-called $\JR$-number is $\JR(\OO_K) = 4$, i.e., the minimal possible value, while the totally real fields  $K$ considered in \cite{mruv, Springer20} have the maximal possible value $\JR(\OO_K) = \infty$; see Section \ref{sec:dec_tr} for definitions and notation. We also provide a version of our undecidability result, namely Theorem \ref{thm:undec_gen}, which does not place a restriction on the value of $\JR(\OO_K)$.

\subsection{Proof method: Leveraging unit groups}
\label{sec:proof_method}
Within the body of literature concerning undecidability results for rings of integers of algebraic extensions of $\QQ$, the totally real subfields of $\Qbar$ and their totally imaginary quadratic extensions have received special attention. There are multiple methods which can be used in this context, but we focus on the work stemming from J. Robinson for now, and delay the discussion of elliptic curves until the following section.

J. Robinson \cite{jr} provided a general sufficient condition for proving the undecidability of a ring of algebraic integers $\OO$, and we follow Videla \cite{vid-cycl} by using an improved version credited to Henson \cite{vdd}: If there is a parametrized family $\curlyF$ of definable subsets of $\OO$ which contains finite sets of arbitrarily large cardinality, then the first-order theory of $\OO$ is undecidable.
This sufficient condition, along with a theorem of Siegel \cite{siegel}, leads to a strategy that can be applied to $\OO_K$ when $K$ is a totally real field. Indeed, we use the parametrized family of subsets $\{X_t\}_{t\in \QQ}$, where $X_t$ contains the elements $\alpha\in \OO_K$ whose conjugates all lie in the real interval~$(0,t)$. Determining whether $\{X_t\}_{t\in \QQ}$ contains finite sets of arbitrarily large size is related to the $\JR$-number of $\OO_K$; see Section \ref{sec:dec_tr} for more details. This strategy enabled J. Robinson to prove the undecidability of the rings of integers of both $\Qtr$ and $\Q^{(2)}_{\tr} = \QQ(\{\sqrt n : n\in \NN\})$.

After using this method to prove $\OO_K$ is undecidable for a totally real field $K$, we wish to do the same for the rings of integers $\OO_L$ for all totally imaginary quadratic extensions $L$~of~$K$.
To extend the $\JR$-number argument to such a field, a key ingredient is the fact that $[\OO_L^\times : \mu(L) \OO_K^\times]\leq 2$ where $\mu(L)$ is the set of roots of unity in $L$. 
Therefore, if $N = \#\mu(L)< \infty$ is finite, then $(\OO_L^\times)^{2N}\subseteq K$ is an existentially definable totally real subset of $\OO_L$. 
By using sums and difference of powers of units, we can produce a useful parametrized family of subsets of $\OO_L$ analogous to the sets $X_t$ defined above. Undecidability is thereby proved for $\OO_L$ in many cases, 
 including $L = \QQ^{(2)}$ or more generally $L = \QQ^{(d)}_{\ab}$ for $d\geq 2$; see \cite{mruv, Springer20}.

However, the assumption $\#\mu(L) < \infty$ above implies that this method cannot yet apply when $L = \Qtr(i)$.
 To remedy this, we choose to instead work with  the unit group $R^\times$ of an existentially definable non-maximal order $R\subseteq \OO_L$ with $\mu(R) = \{\pm1\}$ trivial.
 In this case, we can show that $(R^\times)^2$ is totally real set. 
 By  choosing a suitable ring $R\subseteq \OO_L$, we obtain the foundation for developing a unit group-based argument similar to the ones appearing in \cite{mruv, Springer20} which allows $L$ to contain infinitely many roots of unity.

\subsection{A comparison with abelian varieties}
\label{sec:AV}
We conclude this section by comparing our main theorems to similar results which leverage elliptic curves instead of unit groups. There are many papers which use of elliptic curves, or abelian varieties in general, to prove various undecidability results, including \cite{Denef80, vid-cycl, Poonen02, MRS22}. We refer to \cite{Shlap09} for additional background.  As an example which is relevant to the consideration of totally imaginary extensions of totally real fields, consider the following theorem of Shlapentokh.

\begin{theorem}[Main Theorem B, \cite{Shlap09}]
\label{thm:Shlap09}
Let $K$ be a totally real algebraic extension of $\QQ$ which has a totally real extension of degree $2$, and let $K'$ be a finite extension of $K$ such that there exists an elliptic curve $E$ defined over $K'$ with $E(K')$ finitely generated and of positive rank. If $L$ is a quadratic totally imaginary extension of $K$, then $\ZZ$ is existentially definable in the ring of integers $\OO_K$ and Hilbert's Tenth Problem is unsolvable over $\OO_K$.
 \end{theorem}

Although Theorems \ref{thm:intro_undec_JR4} and \ref{thm:Shlap09} are similar insofar as they both apply in the context of totally imaginary quadratic extensions of totally real subfields of $\Qbar$, it is instructive to also note the differences and complementary strengths. Heuristically, it is easiest to apply the elliptic curve-based methods to relatively ``small'' algebraic extensions of $\QQ$ over which it is easy to find elliptic curves with a finitely generated group of rational points. In contrast, the unit group-based methods work best for ``bigger'' fields in which there is an abundance of units available for manipulation. We can make this more precise with a couple of examples.

The prototypical example of an infinite algebraic extension $K\supseteq \QQ$ which satisfies the hypotheses of Theorem \ref{thm:Shlap09} is a $\ZZ_p$-extension of $\QQ$; see \cite[\S10]{Shlap09}. 
These $\ZZ_p$-extensions are difficult to handle with the unit group-based methods because there is at most one subextension of $K$ of any given degree, hence a paucity of units. We also note that Theorem~\ref{thm:Shlap09} proves that Hilbert's Tenth Problem is unsolvable, rather than only showing that the first-order theory is undecidable, and it also applies when $[K: \QQ]< \infty$.

On the other hand, when deploying elliptic curves, we emphasize that it is necessary for the group of rational points to be a finitely generated group. This is a core requirement of the proof method, and  the restriction would remain for any straightforward variant or generalization which uses abelian varieties instead of elliptic curves, such as \cite[Theorem~1.1]{MRS22}.  Therefore, the following theorem of Fehm and Petersen shows that this general method is not useable when the field $K$ is \emph{large}, in the sense of Pop \cite{Pop96}.

\begin{theorem}[Theorem 1.2, \cite{FP10}]
\label{thm:inf_rank}
If $L\subseteq \Qbar$ is a large field and $A/L$ is an abelian variety, then $A(L)$ has infinite rank.
 \end{theorem}
The first proof of this theorem in the case of elliptic curves is credited to Tamagawa; see Kobayashi \cite[Proposition 1]{Kobayashi06}. We refer to \cite{FJ74,LR08, MR18, Petersen06} for some additional results on finitely and non-finitely generated groups of rational points.
 
 Because any algebraic extension of a large field is itself large, Theorem~\ref{thm:inf_rank} shows that if $K\subseteq \Qbar$ is a large totally real field, then abelian variety-based methods such as Theorem~\ref{thm:Shlap09} cannot handle $K$ or its quadratic totally imaginary extensions. However, it is clear that the unit group-based methods presented in this paper can work for large fields because the field $ \Qtr$ is large.
 It is also conjectured that $\Qab$ is a large field (see \cite[\S3]{BSF13} for background and an overview), and 
 it is easy to construct extensions $L \supseteq \Qab$ which are covered by Theorem~\ref{thm:intro_undec_JR4}.
 It would be interesting to determine whether or not the methods of this paper can be refined to prove the undecidability of the ring of integers of $\Qab$ itself.

\subsection*{Acknowledgments}
The author thanks Arno Fehm for assistance with references, and Aharon Razon for sharing Remark \ref{rem:Razon}. 
Additional thanks to Kirsten Eisentr\"ager, Jochen Koenigsmann, Alexandra Shlapentokh, and Carlos Videla for their helpful comments.
This work was supported by the Additional Funding Programme for Mathematical Sciences, delivered by EPSRC (EP/V521917/1) and the Heilbronn Institute for Mathematical Research.

 \section{Non-maximal orders and units}

 As indicated in Section \ref{sec:proof_method}, given a totally real field $K$ and a totally imaginary quadratic extension $L$, we want a subring $R\subseteq \OO_L$ which does not contain any nontrivial roots of unity. We start by defining the desired ring, then proceed to analyze its group of units.
 
 \subsection{A useful non-maximal subring}
 \begin{definition} 
 Given an integer $m\geq 1$ and a field $L\subseteq \Qbar$, let $R_{m, L}$ be the subset of $\OO_L$ defined by 
 the positive existential formula $\varphi_m(x)$, given as follows:
$$
	\exists a \ \ (x - ma)(x - 1 - ma)\dots (x - (m-1) - ma) = 0
$$
  \end{definition}
 The following properties of $R_{m,L}$ are immediate from the definition.
 
\begin{lemma} 
\label{lem:subring}
Let $L\subseteq \Qbar$ be a field and let $m\geq 1$ be an integer.
\begin{enumerate}[(a)]
	\item If $K\subseteq L$ is a subfield, then $R_{m,L}\cap K = R_{m,K}$.
	\item $R_{m,L}$ is an existentially definable subring of $\OO_L$ which contains $m\OO_L$.
	\item $R_{m,L}\subseteq \OO_L$ is non-maximal if and only if $m\geq 2$ and $L \neq \QQ$. 
\end{enumerate}   
 \end{lemma}
  \begin{proof} 
By the definition of $R_{m,L}$ and the formula $\varphi_m(x)$ above, the first claim is obvious.
Moreover, $R_{m,L}$ is the preimage of $\ZZ/m\ZZ$ under the natural surjective map $\OO_L \to \OO_L/m\OO_L$. Therefore, $R_{m,L}$ is a subring of $\OO_L$ and non-maximal precisely when $\ZZ/m\ZZ \neq \OO_L/m\OO_L$, i.e., when $m \geq 2$ and $L\neq \QQ$. 
  \end{proof}

Before we analyze the unit group $R_{m,L}^\times$, we recall some general elementary facts. In essence, this theorem clarifies that any ring containing an algebraic unit $u$ also contains $u^{-1}$, and computes the rank of the unit groups of any ring of algebraic integers. Given a set $S\subseteq \Qbar$, let $\mu(S)$ denote the set of roots of unity contained in $S$.
  
  \begin{theorem} 
  \label{thm:gen_non-max_units}
  Let  $L\subseteq \Qbar$ be a number field with $r$ real and $2s$ imaginary embeddings.
  \begin{enumerate}[(a)]
	  \item If $u\in \OO_L^\times$, then $u^{-1}\in \ZZ[u]$.
	  \item If $R\subseteq \OO_L$ is any subring, then $R^\times = \OO_L^\times \cap R$.
	  \item If $\OO\subseteq \OO_L$ is any suborder, then $\OO^\times \cong \mu(\OO) \times \ZZ^{r + s - 1}$.
  \end{enumerate} 
   \end{theorem}
    \begin{proof} 
   An algebraic unit $u$ has minimal polynomial $m(x) = x^n + c_{n-1}x^{n-1} \dots + c_1x \pm 1$ with integer coefficients. Evaluating this polynomial at $u$, along with rearranging terms, shows that $u(u^{n-1} + c_{n-1}u^{n-2}+\dots + c_1) = \pm1$, which proves the first claim. The second claim follows immediately from the first.
   
By Dirichlet's unit theorem, $\OO_L^\times \cong \mu(L)\times \ZZ^{r + s - 1}$.     
Thus, we only need to check the rank of $\OO^\times$.
Because $\OO\subseteq \OO_L$ is a suborder, the index $N = [\OO_L : \OO] $ is finite and $N\OO_L \subseteq \OO$. 
If $u\in \OO_L^\times$ then $u$ reduces modulo $N$ to an element of $(\OO_L/N\OO_L)^\times$, 
and $u^{t} \equiv 1\bmod N\OO_L$ for some $1\leq t\leq \#(\OO_L/N\OO_L)^\times < N^{[L : \QQ]}$. 
We conclude $u^t \in \OO^\times$. 
Therefore, $\OO^\times$ has the same rank as $\OO_L^\times$ and we are done.
    \end{proof}
 
As an application, this shows the structure of the unit group of the ring $R_{m,L}$. 
\begin{prop} 
\label{prop:rank}
Let $m\geq 2$ be an integer.  If $L$ is a number field with $r$ real embeddings and $2s$ imaginary embeddings, then 
$$
	R_{m,L}^\times \cong \{\pm 1\} \times \ZZ^{r +s-1}.
$$
  In particular, the only roots of unity contained in $R_{m,L}$ are trivial.
 \end{prop}
  \begin{proof} 
By Theorem \ref{thm:gen_non-max_units}, we only need to show that $R_{m,L}$ has no roots of unity other than $\pm1$. If $\zeta \in R_{m,L}^\times$ is a nontrivial root of unity, then $L$ contains the nontrivial cyclotomic subfield $L_0 = \QQ(\zeta) \supsetneq \QQ$, and thus  $R_{m,L_0} = R_{m,L}\cap L_0\supseteq \ZZ[\zeta]$ by Lemma \ref{lem:subring}. But $\ZZ[\zeta]$ is the maximal order of $L_0$, while $R_{m,L_0}$ is a non-maximal order, which is a contradiction.
  \end{proof}

  \subsection{Totally imaginary extensions of totally real fields}
  
  We now restrict our attention to the main focus of this paper: totally imaginary quadratic extensions of totally real fields.
The following is a generalization of \cite[Theorem 4.12]{Washington} to the case of non-maximal orders. Given a field $L$ and a subring $\OO\subseteq L$, we write $\mu(\OO)$ for the set of roots of unity in $\OO$.

  \begin{theorem} 
\label{thm:inside}
Let $K$ be a totally real field, let $L$ be a totally imaginary quadratic extension, and let $\OO\subseteq L$ be an order which is stable under complex conjugation. If $u\in \OO^\times$, then its complex conjugate is $\overline u = \zeta u$ where $\zeta\in \mu(\OO)$.  In particular,  writing $\OO_{\tr} = \OO\cap K$,
$$
	[\OO^\times : \mu(\OO) \cdot \OO_{\tr}^\times] \leq 2.
$$
 \end{theorem}
  \begin{proof} 
Notice that every conjugate of $u/\overline u$ has absolute value 1 because $\sigma(\overline u) = \overline\sigma(u)$ for every embedding $\sigma: L\hookrightarrow \CC$; see  \cite[p.39]{Washington}.
 Therefore, $\frac{u}{\overline u}$ is a root of unity \cite[Lemma 1.6]{Washington}.
 In other words, $\frac{u}{\overline u}\in \mu(L)\cap \OO = \mu(\OO)$. 
 
 Thus, we define a group homomorphism $\varphi: \OO^\times \to \mu(\OO)/\mu(\OO)^2$ induced by $u \mapsto u/\overline u$ and compute its kernel. 
 First, we check that if $u = \zeta u_1$ for $\zeta \in \mu(\OO)$ and $u_1\in \OO^\times \cap K$, then $\varphi(u) = u/\overline u = \zeta u_1/\overline{\zeta u_1} = \zeta^2$.
 Conversely, if $u\in \OO^\times$ and $\varphi(u) = u/\overline u=  \zeta^2$, then we rearrange to see $\overline \zeta u = \zeta \overline u\in \OO_{\tr}$ is totally real.
  Thus, $\ker \varphi =  \mu(\OO) \cdot \OO_{\tr}^\times$. Since we know the index $[\mu(\OO) : \mu(\OO)^2] \leq 2$, the proof is done.
  \end{proof}

We now wish to apply this theorem to the non-maximal subrings defined in the previous section. In our context, this theorem is helpful because we define subrings with no nontrivial roots of unity. Before we move on, we note that the application of Theorem \ref{thm:inside} in the case when $\OO = \OO_L$ is a central ingredient in \cite{mruv, Springer20}. The following corollary was  also inspired by an analogous fact for maximal orders \cite[Proof of Proposition 1.5]{Washington}.

  \begin{remark} 
   It is easy to show, e.g., using Magma \cite{Magma}, that it is possible to have $R_{2,L}^\times\not\subseteq K$ in Corollary \ref{cor:inside}. Indeed, this occurs for the cyclotomic field $L= \QQ(\zeta_{15})$. Thus, squaring the group of units is necessary in the statement of the corollary when $m = 2$.
   \end{remark}

\begin{corollary} 
\label{cor:inside}
Let $K$ be a totally real field with totally imaginary quadratic extension $L$.
\begin{enumerate}[(a)]
	\item $(R_{2,L}^\times)^2\subseteq R_{2,K}^\times$
	\item If $m \geq 3$, then $R_{m,L}^\times=R_{m,K}^\times$.
\end{enumerate}
 \end{corollary}
  \begin{proof} 
Let $u\in R_{m,L}^\times$.  Clearly $R_{m,L}$ is stable under complex conjugation by definition. Combining Proposition \ref{prop:rank} and Theorem \ref{thm:inside}, either $\overline u = u$ or $\overline u = -u$. 
To finish the proof, assume $m\geq 3$ and suppose that $\overline u = -u$. Writing $u = mb+ j$ for $b\in \OO_L$ and $j\in \ZZ$ we have $\overline u = m\overline b + j$.  Therefore, 
$$
	2u = u - \overline u = mb - m\overline b\in m\OO_L.
$$
However, this implies that $m$ divides $2$ because $u$ is a unit, and this is impossible because $m \geq 3$.
We conclude that $\overline u = u$ is totally real.
  \end{proof}

  \section{Defining the totally real subring.}
  
  We now show how to deploy the unit groups described in the previous section. To do this, we explicitly write certain totally real algebraic integers as the sum of units.
  
     \begin{lemma} 
       \label{lem:squares}
     Let $K_0\subseteq K_1$ be totally real fields. If $a = 2(2d + 1)$ for some $d\in \OO_{K_0}$ such that 
     $\sqrt{d^2 + d}\in K_1$, then there is a unit $u\in R_{2,K_1}^\times$ such that
       $
       	a = u + 1/u.
       $
       In particular, 
       $$
       	a^2 = u^2 +1/u^2 +2.
       $$
        \end{lemma}
         \begin{proof} 
         Define $u_j = 2b_j + 1\in R_{2,K_1}$ where
       $
       	b_1 = d + \sqrt{d^2 + d}
       $
       and       $
       	b_2 = d - \sqrt{d^2 + d}
       $.       
       It is easy to compute directly that $u_1$ and $u_2$ are inverses of each other, and hence units from $R_{2,K_1}$:
       \begin{align*}
       	u_1u_2 
		= (2b_1 + 1)(2b_2 + 1) 
		&= (2d +1 + 2\sqrt{d^2 + d})(2d+1 -2 \sqrt{d^2 + d })\\
		&=(2d + 1)^2 - 4(d^2 + d)
		=1.
       \end{align*} 
       Moreover,
       $
	u_1 +u_2 = 2(b_1 + b_2+1) = 2(2d + 1) = a.
       $
       We take $u = u_1$, in which case $1/u = u_2$. It follows immediately that
       $
       	a^2 = (u + 1/u)^2 =  u^2 + 1/u^2 + 2.
       $
         \end{proof}

      \begin{lemma} 
      \label{lem:all_32}
     Let $K_0$ be a totally real field.  If $d\in \OO_{K_0}$ and all conjugates of $d$ are outside the open interval $(-1,1)\subseteq \RR$, then 
       $$
       	32d =   u^2 + 1/u^2 - v^2 - 1/v^2
	$$
	for units $u, v\in R_{2,{K_1}}^\times$ where $K_1 = K_0(\sqrt{d^2 + d}, \sqrt{(d-1)^2 + (d-1)})$ is totally real.
       \end{lemma}
        \begin{proof} 
        It is easy to see that the values of the function $f(x) = x^2 + x$ are negative precisely when $x\in (-1,0)$, so the element $d^2 + d$ is totally nonnegative if and only if all conjugates of $d$ lie outside of $(-1,0)$.
        Similarly, $(d-1)^2 + (d-1)$ is totally nonnegative if and only if all conjugates of $d$ lie outside of $(0,1)$.
        Therefore, if the conjugates of $d$ lie outside of $(-1,1)$, then elements $d^2 + d$ and $(d-1)^2 + (d-1)$ are both totally nonnegative, so their square roots $\sqrt{d^2 + d}$ and $ \sqrt{(d-1)^2 + (d-1)}$ are totally real.
        Thus, applying Lemma \ref{lem:squares} twice,
        $$
        32d = 4(2d+1)^2 - 4(2(d-1) + 1)^2 = u^2 + 1/u^2 - v^2 - 1/v^2.
        $$
        for some $u,v\in R_{2,K_1}^\times$.
        \end{proof}

Now we are ready to define the desired totally real subsets $X$ within the totally imaginary fields.  This can be seen as a generalization of \cite[Lemma 7]{mruv} and \cite[Lemma 2.7]{Springer20}.
 \begin{theorem} 
 \label{thm:def_tr_loose}
Let $K_0$ be a totally real field and let $K_1$ be the maximal totally real subfield of $K_0^{(2)}$.   If $K\supseteq K_1$ is a totally real field, then for every totally imaginary quadratic extension $L$ of $K$, there is an existentially definable subset $X\subseteq \OO_L$ satisfying
$$
	\OO_{K_0}\subseteq X \subseteq \OO_K.	
$$
  \end{theorem}
   \begin{proof} 
  We write
  $$
  	X_0 = \{u_1^2 + u_2^2 - u_3^2 - u_4^2 : u_1, \dots, u_4 \in R_{2,L}^\times\}.
  $$
  We have $X_0 \subseteq \OO_K$ by Corollary \ref{cor:inside}, and the set
  $$
  	X_1 = \{d \in \OO_L : 32d\in X_0\}.
  $$
 contains all elements $d\in \OO_{K_0}$ whose conjugates are all outside $(-1,1)$ by Lemma \ref{lem:all_32}.  
 We repeat the same trick from before to finish the proof: If $d\in \OO_{K_0}$, then $(d-1)^2$ and $(d+1)^2$ are totally nonnegative. 
 In particular, the elements $(d-1)^2 + 1$ and $(d+1)^2 +1$ are contained in $X_1$ because their conjugates lie inside $[1,\infty)$, and outside $(-1,1)$.
 Thus, we may write $4d$ as the difference of these two elements of $X_1$, namely
 $
 	 4d = (d+1)^2+1 - [(d-1)^2+1].
 $
 
In summary,
$
	\OO_{K_0} 
		\subseteq X
		\subseteq \OO_K,
$
where
$$
	X =  \{\alpha \in \OO_L : \exists x_1,x_2\in X_1, \  4\alpha = x_1 - x_2 \}.
$$
Because the ring $R_{2,L}$ is existentially definable in $\OO_L$, this completes the proof.
   \end{proof}
   
   \begin{remark} 
The definition of the subset $X$ in Theorem~\ref{thm:def_tr_loose} is uniform in the sense that the formula defining $X$ does not depend on the choices of $K_0, K_1, K$, or $L$. 
    \end{remark}

By applying the previous theorem in the case where $K_ 0 = K$, we find that the ring of integers $\OO_K$ itself is existentially definable in $\OO_L$, as written in the following theorem. Recall that $K$ is \emph{closed under square roots} if $\alpha\in K$ whenever $\alpha\in \Qbar$ is totally real and $\alpha^2 \in K$.
   
   \begin{theorem} 
   \label{thm:def_tr}
    If $K$ is a totally real field which is closed under square roots and $L$ is any quadratic totally imaginary extension of $K$, then $\OO_K$ is existentially definable in $\OO_L$.
    \end{theorem}

     Since $\Qtr$ is closed under square roots, we obtain the definability portion of Theorem \ref{thm:intro_motivation}.
     \begin{corollary} 
     \label{cor:Ztr}
     $\Ztr$ is existentially definable in the ring of integers of $\Qtr(i)$.
      \end{corollary}

   Throughout this section, we have exploited the fact that, within the fields $K$ that we consider, many algebraic integers can be written as the sum of a bounded number of units.
      For contrast, note that Frey and Jarden showed the following when $F$ is a number field: There is no integer $n\geq 1$ such that every element of $\OO_F$ is the sum of at most $n$ units in $\OO_F^\times$ \cite[Theorem~1]{JN07}. In the same paper, they construct many infinite algebraic extensions of $\QQ$ in which every element of the ring of integers is the sum of at most 2 units; see \cite[Theorem~8]{JN07}. Although the fields they consider are not used in this paper, their complementary results illustrate how we are  exploiting a phenomenon that exists only for (some) infinite algebraic extensions of $\QQ$, and not for number fields.

 \section{Undecidability}
 \label{sec:undec}
 
The definability results in the preceding section have immediate consequences for decidability. Namely, in the context of Theorem \ref{thm:def_tr}, if the existential or first-order theory of $\OO_K$ is undecidable then the same is true for $\OO_L$. However, this is not the end of the story because the implications for decidability go beyond the hypotheses of the theorem. Indeed, we can deduce a much more general undecidability result with a further development of the ideas used in previous unit group-based methods \cite{mruv, Springer20}. Specifically, we only need to define ``enough'' of the maximal totally real subring to realize the $\JR$-number.

%
%
%
 
 To start, we recall the following lemma, which is C.W. Henson's extension  \cite[\S3.3]{vdd} of a result of J. Robinson \cite[Theorem 2]{jr}; see also the presentation in \cite[Lemma~2.2]{mruv}. This will be the basis for all of our undecidability results.

\begin{lemma} 
\label{lem:henson}
Let $\OO$ be a ring of algebraic integers. If there is a family $\mathcal F$ of subsets of $\OO$, parametrized by an $\mathcal L_{ring}$-formula, which contains finite sets of arbitrarily large cardinality, then $\OO$ has undecidable first-order theory.
 \end{lemma}

  \subsection{Undecidability for totally real $\OO_K$.}
  \label{sec:dec_tr}
Before considering totally imaginary fields, we briefly sketch the $\JR$-number method for proving the undecidability of the ring of integers of a totally real field; see \cite{jr, vv-nested, vv-northcott} for more details. 

\begin{definition} 
 Given a totally real number $\alpha\in \Qbar$ and $t\in \RR$, we write $0\ll \alpha \ll t$ if every conjugate of $\alpha$ lies in the interval $(0,t)$. For a totally real subset $X\subseteq \Qbar$, write $X_t = \{\alpha \in X : 0\ll \alpha\ll t\}$.
The $\JR$-number of $X$ is
$$
	\JR(X) = \inf\{t\in \RR : \#X_t =\infty\}.
$$
 \end{definition}

In particular, when $K$ is totally real, J. Robinson \cite{jr} observed that if $\JR(\OO_K)$ is either $\infty$ or a minimum, then Lemma \ref{lem:henson} implies that $\OO_K$ is undecidable.
 Indeed, $\{X_t\}_{t\in \QQ}$ can be realized as a parametrized family of existentially definable sets when $X = \OO_K$ and $t\in\QQ$, as written below.
 This follows from Siegel's proof that the totally nonnegative elements of $K$ are precisely the elements which are sums of four squares \cite{siegel}. 
Consequently, J. Robinson deduced that the ring of totally real algebraic integers $\Ztr$ is undecidable.

\begin{theorem}
\label{thm:def_siegel}
Let $K$ be a totally real field. If $X = \OO_K$ and $t =\frac{a}{b}\in \QQ$, then $X_t$ is defined in $\OO_K$ by the formula:
    $$
    	\exists y_0,\dots, y_8\in \OO_K [bxy_0^2\neq 0 \wedge bxy_0^2 \neq a \wedge xy_0^2 = y_1^2 + \dots + y_4^2 \wedge (a-bxy_0^2) = y_5^2 +\dots + y_8^2]
    $$
 \end{theorem}
 
 Note that, in general, algebraic integers might be written as the sum of squares of algebraic non-integers, so $y_0$ plays the role of a common denominator.

  \subsection{Undecidability for totally imaginary $\OO_L$.}
  We are now ready to deploy a modified version of the $\JR$-method for totally imaginary fields; compare with \cite{mruv, Springer20}.
    
 \begin{theorem} 
 \label{thm:undec_gen}
 Let $K_0$ be a totally real field for which $\JR(\OO_{K_0})$ is either an attained minimum or $\infty$, and let $K_1$ be the maximal totally real subfield of $K_0^{(2)}$.  If $K\supseteq K_1$ is any totally real field with $\JR(\OO_K) = \JR(\OO_{K_0})$, and $L$ is any totally imaginary quadratic extension of $K$, then the first order theory of $\OO_L$ is undecidable.
  \end{theorem}
   \begin{proof} 
By Theorem \ref{thm:def_tr_loose}, there is an existentially definable subset $X  \subseteq \OO_L$ satisfying
    $$
    	\OO_{K_0} \subseteq X \subseteq \OO_K.
    $$
    Therefore, the parametrized formula $\phi_X(x; a,b)$ defined by    $$
    	\exists y_0,\dots, y_8\in X  [bxy_0^2\neq 0 \wedge bxy_0^2 \neq a \wedge xy_0^2 = y_1^2 + \dots + y_4^2 \wedge (a-bxy_0^2) = y_5^2 +\dots + y_8^2]
    $$
    defines sets which satisfy the containments
    \begin{equation*}
    	\left\{x \in \OO_{K_0} : 0\ll x\ll \frac{a}{b}\right\}
    	\subseteq \{ x\in \OO_L : \phi_X(x;a,b) \}
	\subseteq \left\{x\in \OO_K : 0\ll x\ll \frac{a}{b}\right\}.
    \end{equation*} 
    The containments follow from Theorem \ref{thm:def_siegel}.
    Because $\JR(\OO_K) = \JR(\OO_{K_0})$ is either infinite or an attained minimum, the sets on the lefthand and righthand sides above are finite sets of arbitrarily large size as $\frac{a}{b}$ ranges over positive rational numbers approaching $\JR(\OO_K)$ from the left.  Therefore, by Lemma \ref{lem:henson}, we have proved first-order undecidability.
   \end{proof}
   
   There are two extremes for the JR-number of a ring of totally real algebraic integers, namely 4 and $\infty$.  For example, the fields $K_0 = \QQ$ and $K_1 = \QQ(\{\sqrt n : n\geq 1\})$ have $\JR(\OO_{K_0}) = \JR(\OO_{K_1}) = \infty$. Applying Theorem \ref{thm:undec_gen} in this case leads to a recovery of some of the results that appeared in \cite{mruv, Springer20}.  At the other extreme, we write the following.
    
    \begin{theorem} 
    \label{thm:undec_JR4}
    Let $S$ be any infinite set of positive integers, let $K_0 = \QQ(\{\zeta_n + \overline\zeta_n : n\in S\})$, and let $K_1$ be the maximal totally real subfield of $K_0^{(2)}$. If $K$ is any totally real extension of $K_1$ and $L$ is any totally imaginary quadratic extension of $K$, then the first order theory of $\OO_L$ is undecidable.
     \end{theorem}
      \begin{proof} 
     Under the given hypotheses, $\JR(\OO_{K_0}) = 4$ is the smallest possible $\JR$-number and is realized as a minimum. Indeed, we have $0\ll \zeta_n + \overline \zeta_n + 2\ll 4$ for all $n\in S$, and these are the only totally real elements of $\Qbar$ with this property; see Kronecker \cite{Kronecker} and also \cite{RMR}.
     Therefore, $\JR(\OO_{K_0}) = \JR(\OO_K)$ is clear and Theorem \ref{thm:undec_gen} applies.
      \end{proof}

In recent years, examples of totally real fields $K$ have been discovered for which $\JR(\OO_K)$ is neither $4$ nor $\infty$, i.e., $\JR(\OO_K)$ is non-extremal; see \cite{castillo-thesis, cfvv, gr, vv-nested}. A better understanding of the behavior of $\JR$-numbers under field extension would be required before applying Theorem \ref{thm:undec_gen} to fields $K_0$ for which $\JR(\OO_{K_0}) \in (4,\infty)$ is non-extremal.


\begin{thebibliography}{MRUV20}

\bibitem[BCP97]{Magma}
Wieb Bosma, John Cannon, and Catherine Playoust.
\newblock The {M}agma algebra system. {I}. {T}he user language.
\newblock {\em J. Symbolic Comput.}, 24(3-4):235--265, 1997.
\newblock Computational algebra and number theory (London, 1993).

\bibitem[BSF13]{BSF13}
Lior Bary-Soroker and Arno Fehm.
\newblock Open problems in the theory of ample fields.
\newblock In {\em Geometric and differential {G}alois theories}, volume~27 of
  {\em S\'{e}min. Congr.}, pages 1--11. Soc. Math. France, Paris, 2013.

\bibitem[{Cas}18]{castillo-thesis}
Marianela {Castillo Fern\'andez}.
\newblock {\em On the Julia Robinson number of rings of totally real algebraic
  integers in some towers of Nested Square Roots}.
\newblock PhD thesis, Universidad de Concepci\'on, 2018.
\newblock URL: http://repositorio.udec.cl/handle/11594/3003.

\bibitem[Cha19]{Chatzidakis19}
Zo\'{e} Chatzidakis.
\newblock Amalgamation of types in pseudo-algebraically closed fields and
  applications.
\newblock {\em J. Math. Log.}, 19(2):1950006, 28, 2019.

\bibitem[CVV20]{cfvv}
Marianela {Castillo Fern\'andez}, Xavier Vidaux, and Carlos~R. Videla.
\newblock Julia {R}obinson numbers and arithmetical dynamic of quadratic
  polynomials.
\newblock {\em Indiana Univ. Math. J.}, 69(3):873--885, 2020.

\bibitem[Dar00]{Darniere00}
L.~Darni\`ere.
\newblock Decidability and local-global principles.
\newblock In {\em Hilbert's tenth problem: relations with arithmetic and
  algebraic geometry ({G}hent, 1999)}, volume 270 of {\em Contemp. Math.},
  pages 145--167. Amer. Math. Soc., Providence, RI, 2000.

\bibitem[Den80]{Denef80}
J.~Denef.
\newblock Diophantine sets over algebraic integer rings. {II}.
\newblock {\em Trans. Amer. Math. Soc.}, 257(1):227--236, 1980.

\bibitem[DF21]{DF21}
Philip Dittmann and Arno Fehm.
\newblock Nondefinability of rings of integers in most algebraic fields.
\newblock {\em Notre Dame J. Form. Log.}, 62(3):589--592, 2021.

\bibitem[FHV94]{fhv}
Michael~D. Fried, Dan Haran, and Helmut V\"{o}lklein.
\newblock Real {H}ilbertianity and the field of totally real numbers.
\newblock In {\em Arithmetic geometry ({T}empe, {AZ}, 1993)}, volume 174 of
  {\em Contemp. Math.}, pages 1--34. Amer. Math. Soc., Providence, RI, 1994.

\bibitem[FJ74]{FJ74}
Gerhard Frey and Moshe Jarden.
\newblock Approximation theory and the rank of abelian varieties over large
  algebraic fields.
\newblock {\em Proc. London Math. Soc. (3)}, 28:112--128, 1974.

\bibitem[FP10]{FP10}
Arno Fehm and Sebastian Petersen.
\newblock On the rank of abelian varieties over ample fields.
\newblock {\em Int. J. Number Theory}, 6(3):579--586, 2010.

\bibitem[GR17]{gr}
Pierre Gillibert and Gabriele Ranieri.
\newblock Julia robinson's numbers.
\newblock {\em arXiv e-prints}, Oct 2017.

\bibitem[Jar11]{Jarden11}
Moshe Jarden.
\newblock {\em Algebraic patching}.
\newblock Springer Monographs in Mathematics. Springer, Heidelberg, 2011.

\bibitem[JN07]{JN07}
Moshe Jarden and W\l adys\l~aw Narkiewicz.
\newblock On sums of units.
\newblock {\em Monatsh. Math.}, 150(4):327--332, 2007.

\bibitem[Kob06]{Kobayashi06}
Emi Kobayashi.
\newblock A remark on the {M}ordell-{W}eil rank of elliptic curves over the
  maximal abelian extension of the rational number field.
\newblock {\em Tokyo J. Math.}, 29(2):295--300, 2006.

\bibitem[Koe14]{Koenigsmann14}
Jochen Koenigsmann.
\newblock Undecidability in number theory.
\newblock In {\em Model theory in algebra, analysis and arithmetic}, volume
  2111 of {\em Lecture Notes in Math.}, pages 159--195. Springer, Heidelberg,
  2014.

\bibitem[Kro57]{Kronecker}
L.~Kronecker.
\newblock Zwei {S}\"{a}tze \"{u}ber {G}leichungen mit ganzzahligen
  {C}oefficienten.
\newblock {\em J. Reine Angew. Math.}, 53:173--175, 1857.

\bibitem[LR08]{LR08}
\'{A}lvaro Lozano-Robledo.
\newblock Ranks of abelian varieties over infinite extensions of the rationals.
\newblock {\em Manuscripta Math.}, 126(3):393--407, 2008.

\bibitem[MR18]{MR18}
Barry Mazur and Karl Rubin.
\newblock Diophantine stability.
\newblock {\em Amer. J. Math.}, 140(3):571--616, 2018.
\newblock With an appendix by Michael Larsen.

\bibitem[MRS22]{MRS22}
Barry {Mazur}, Karl {Rubin}, and Alexandra {Shlapentokh}.
\newblock {Existential definability and diophantine stability}.
\newblock {\em arXiv e-prints}, page arXiv:2208.09963, August 2022.

\bibitem[MRUV20]{mruv}
Carlos Mart\'{\i}nez-Ranero, Javier Utreras, and Carlos~R. Videla.
\newblock Undecidability of {$\Bbb Q^{(2)}$}.
\newblock {\em Proc. Amer. Math. Soc.}, 148(3):961--964, 2020.

\bibitem[Pet06]{Petersen06}
Sebastian Petersen.
\newblock On a question of {F}rey and {J}arden about the rank of abelian
  varieties.
\newblock {\em J. Number Theory}, 120(2):287--302, 2006.

\bibitem[Poo02]{Poonen02}
Bjorn Poonen.
\newblock Using elliptic curves of rank one towards the undecidability of
  {H}ilbert's tenth problem over rings of algebraic integers.
\newblock In {\em Algorithmic number theory ({S}ydney, 2002)}, volume 2369 of
  {\em Lecture Notes in Comput. Sci.}, pages 33--42. Springer, Berlin, 2002.

\bibitem[Pop96]{Pop96}
Florian Pop.
\newblock Embedding problems over large fields.
\newblock {\em Ann. of Math. (2)}, 144(1):1--34, 1996.

\bibitem[Rob62]{jr}
Julia Robinson.
\newblock On the decision problem for algebraic rings.
\newblock In {\em Studies in mathematical analysis and related topics}, pages
  297--304. Stanford Univ. Press, Stanford, Calif, 1962.

\bibitem[Rob64]{RMR}
Raphael~M. Robinson.
\newblock Intervals containing infinitely many sets of conjugate algebraic
  units.
\newblock {\em Ann. of Math. (2)}, 80:411--428, 1964.

\bibitem[Shl09]{Shlap09}
Alexandra Shlapentokh.
\newblock Rings of algebraic numbers in infinite extensions of {$\Bbb Q$} and
  elliptic curves retaining their rank.
\newblock {\em Arch. Math. Logic}, 48(1):77--114, 2009.

\bibitem[Shl18]{shlap}
Alexandra Shlapentokh.
\newblock First-order decidability and definability of integers in infinite
  algebraic extensions of the rational numbers.
\newblock {\em Israel Journal of Mathematics}, 226(2):579--633, 2018.

\bibitem[Sie21]{siegel}
Carl Siegel.
\newblock Darstellung total positiver {Z}ahlen durch {Q}uadrate.
\newblock {\em Math. Z.}, 11(3-4):246--275, 1921.

\bibitem[Spr20]{Springer20}
Caleb Springer.
\newblock Undecidability, unit groups, and some totally imaginary infinite
  extensions of {$\Bbb{Q}$}.
\newblock {\em Proc. Amer. Math. Soc.}, 148(11):4705--4715, 2020.

\bibitem[VDD88]{vdd}
Lou Van Den~Dries.
\newblock Elimination theory for the ring of algebraic integers.
\newblock {\em Journal fur die Reine und Angewandte Mathematik},
  1988(388):189--205, 1988.

\bibitem[Vid00a]{vid}
Carlos~R. Videla.
\newblock Definability of the ring of integers in pro-p galois extensions of
  number fields.
\newblock {\em Israel Journal of Mathematics}, 118(1):1--14, Dec 2000.

\bibitem[Vid00b]{vid-cycl}
Carlos~R. Videla.
\newblock The undecidability of cyclotomic towers.
\newblock {\em Proc. Amer. Math. Soc.}, 128(12):3671--3674, 2000.

\bibitem[VV15a]{vv-nested}
Xavier Vidaux and Carlos~R. Videla.
\newblock Definability of the natural numbers in totally real towers of nested
  square roots.
\newblock {\em Proceedings of the American Mathematical Society},
  143(10):4463--4477, 2015.

\bibitem[VV15b]{vv-northcott}
Xavier Vidaux and Carlos~R. Videla.
\newblock A note on the northcott property and undecidability.
\newblock {\em Bulletin of the London Mathematical Society}, 48(1):58--62,
  2015.

\bibitem[Was97]{Washington}
Lawrence~C. Washington.
\newblock {\em Introduction to cyclotomic fields}, volume~83.
\newblock Springer, New York, 2nd edition, 1997.

\end{thebibliography}
\end{document}